\documentclass{amsart}

\usepackage{graphicx}
\usepackage[latin1]{inputenc}

\newtheorem{thm}{Theorem}[section]
\newtheorem{cor}[thm]{Corollary}
\newtheorem{lem}[thm]{Lemma}

\newtheorem{exa}[thm]{Example}
\theoremstyle{definition}
\newtheorem{dfn}[thm]{Definition}
\theoremstyle{remark}
\newtheorem{rem}[thm]{Remark}
\numberwithin{equation}{section}

\begin{document}

\title[]{Knots in $S_{g} \times S^{1}$ and winding parities}%
\author{S.Kim}

\address{S.Kim, Jilin university}%
\email{ksj19891120@gmail.com}%


\maketitle

\begin{abstract}
A virtual knot, which is one of generalizations of knots in $\mathbb{R}^{3}$ (or $S^{3}$), is, roughly speaking, an embedded circle in thickened surface $S_{g} \times I$. In this talk we will discuss about knots in 3 dimensional $S_{g} \times S^{1}$. We introduce basic notions for knots in $S_{g} \times S^{1}$, for example, diagrams, moves for diagrams and so on. For knots in $S_{g} \times S^{1}$ technically we lose over/under information, but we have information ``how many times a half of the crossing of the knot in $S_{g} \times S^{1}$ rotates along $S^{1}$'', we call it labels of crossings. In this paper we extend this notion more generally and discuss its geometrical meaning. This paper follows from \cite{Kim}.
\end{abstract}

\section{Introduction}
One of generalizations of classical knot theory is {\em virtual knot theory}.
\begin{dfn}
{\em A virtual link} is an equivalence class of virtual link diagrams modulo generalized Reidemeister moves described in Fig~\ref{vir_moves}. That is,

$$\{Virtual~link\} =\{ Virtual~link~diagrams\}/\langle Generalised~Reidemeister~moves \rangle$$
\end{dfn}

\begin{figure}
\centering\includegraphics[width=200pt]{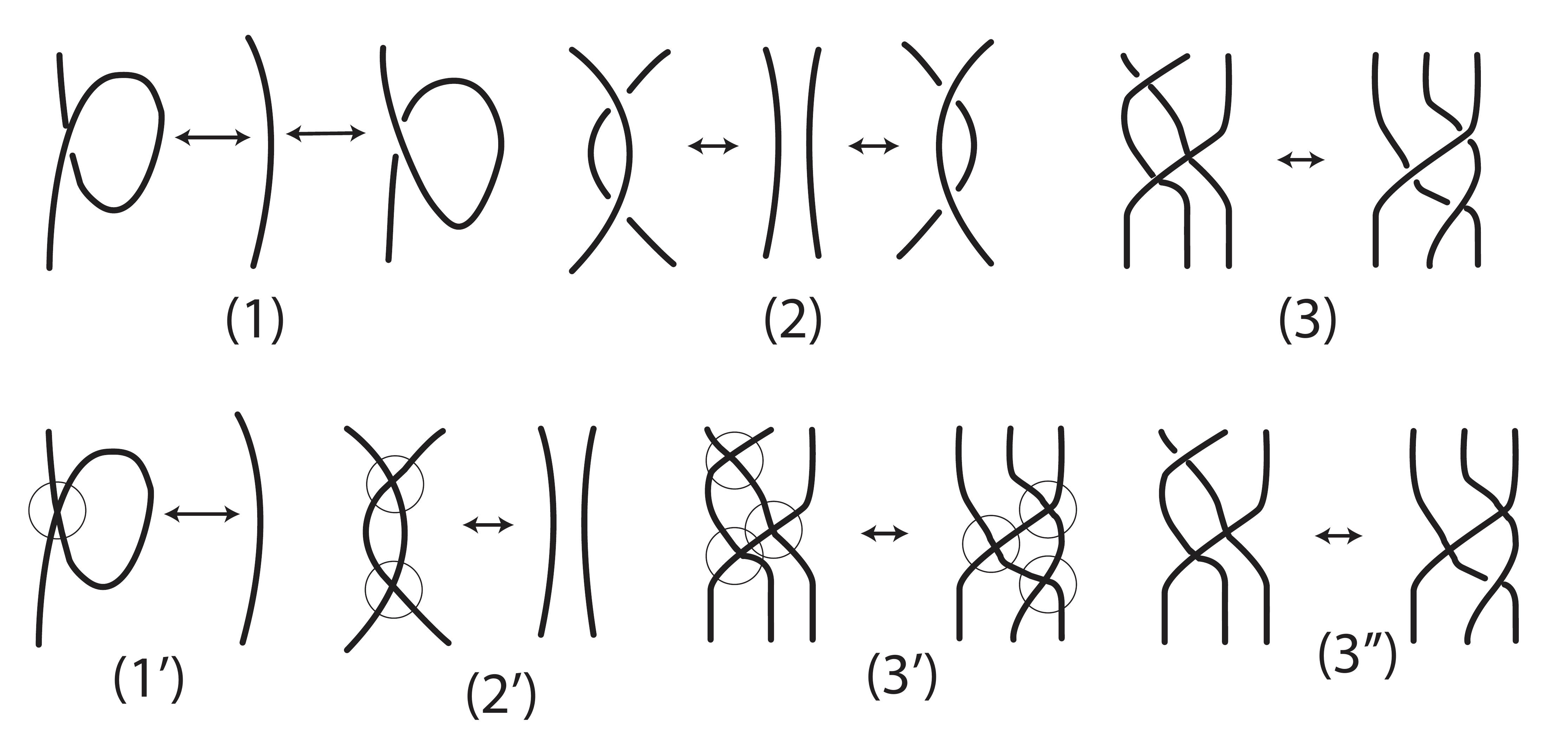} 
\caption{Generalized Reidemeister moves} \label{vir_moves}
 \end{figure}

It is well-known that the virtual links can be considered as links in a thickened surface $S_{g} \times [0,1]$ up to stabil/destabilization, where $S_{g}$ is an orientable surface of genus $g$.

\begin{dfn}
{\em A virtual link} is a smooth embedding $L$ of a disjoint union of $S^{1}$ into $S_{g} \times [0,1]$. Each image of $S^{1}$ is called {\em a component} of $L$. A link of one component is called {\em a virtual knot}.

\end{dfn}

\begin{dfn}
  Let $L$ and $L'$ be two virtual links. If $L'$ can be obtained from $L$ by diffeomorshisms and stabil/destabilizations of $S_{g} \times [0,1]$, then we call $L$ and $L'$ are \textit{equivalent}.
\end{dfn}

In virtual knot theory, by using {\em the parity} defined by V.O. Manturov many invariants for classical knots can be non trivially extended to virtual knots and it gives several interesting geometrical properties, for details, see \cite{Manturov_vir_book}. But, the extended invariants cannot show us new properties of classical knots, because every parity for classical knots is trivial \cite{IlyutkoManturovNikonov}.

In \cite{CM} M. Chrisman and V.O. Manturov studied virtual knots by using 2-component link $K \sqcup J$ with $lk(K,J)=0$, where $J$ is a fibered knot. Roughly speaking, if $J$ is a fibered knot, $S^{3} \backslash N(J)$ is homeomorphic to $\Sigma_{J} \times S^{1}$ where $\Sigma_{J}$ is a Seifert surface of $J$ and $K$ can be considered as a knot in $\Sigma_{J} \times S^{1}$. If $lk(K,J)=0$, then there exists a lifting $\hat{K} \subset \Sigma_{J} \times (0,1) \subset \Sigma_{J} \times [0,1]$ along the covering $p: \Sigma_{J} \times (0,1) \cong \Sigma_{J} \times \mathbb{R} \rightarrow \Sigma_{J} \times S^{1}$ defined by $p(x,r) = (x,e^{2\pi r})$. Then $\hat{K} \subset \Sigma_{J} \times [0,1]$ is placed in a thickened surface, that is, it can be considered as a virtual knot. Moreover, in \cite{CM} it is proved that the $\hat{K}$ is well-defined, that is, if $K \sqcup J$ and $K' \sqcup J'$ are equivalent in $S^{3}$, then $\hat{K}$ and $\hat{K}'$ are equivalent as virtual knots. But, there is a question: what happens if $lk(K,J) \neq 0$?

In \cite{Kim}, the author constructed knots in $S_{g}\times S^{1}$ and local moves.
In \cite{KimSgS1}, the author defined ``labels'' of crossings of knots in $S_{g} \times S^{1}$ and its applications.

This paper is contributed to expand the notion of ``labels'' of crossings of knots in $S_{g} \times S^{1}$. In Section 2, we introduce basic notions of links in $S_{g} \times S^{1}$ and labels of crossings of knots in $S_{g} \times S^{1}$. In Section 3, we define {\em a winding parity} which is a generalization of labels of crossings of knots in $S_{g} \times S^{1}$ defined axiomatically. We introduce examples of winding parities and define an important example, called {\em homological parity} by using 1st homology of ambient space $S_{g}\times S^{1}$.

\section{Links in $S_{g} \times S^{1}$ and its diagrams}

Let $S_{g}$ be an orientable surface of genus $g$. Let us define links in $S_{g} \times S^{1}$ analogously to virtual links by using underlying surfaces as follows:
\begin{dfn}
{\em A link $L$ in $S_{g} \times S^{1}$} is a smooth embedding $L$ of a disjoint union of $S^{1}$ into $S_{g} \times S^{1}$. Each image of $S^{1}$ is called {\em a component} of $L$. A link of one component is called {\em a knot in $S_{g} \times S^{1}$}.

\end{dfn}

\begin{dfn}
  Let $L$ and $L'$ be two links in $ S_{g} \times S^{1}$. If $L'$ can be obtained from $L$ by diffeomorshisms and stabil/destabilization of $S_{g} \times S^{1}$, then we call $L$ and $L'$ are \textit{equivalent}.
\end{dfn}

By the {\em destabilization for $S_{g} \times S^{1}$} we mean the following:
Let $C$ be a non-contractible circle on the surface $S_{g}$ such that there exists a torus $T$ homotopic to the torus $C \times S^{1}$ and not intersecting the link. Then our destabilization is cutting of the manifold $S_{g} \times S^{1}$ along the torus $C \times S^{1}$ and pasting of two newborn components by $D \times S^{1}$.

Now let us construct diagrams for links in $S_{g}\times S^{1}$ on the plane as follows:\\
Let $L$ be an (oriented) link in $S_{g} \times S^{1}$. Assume that counterclockwise orientation is given on $S^{1}$. Suppose that $x_{0} \in S^{1}$ is a point such that $S_{g} \times \{x_{0}\} \cap L(S^{1})$ is a set of finite points with no transversal points. 
\begin{figure}[h]
\begin{center}
 \includegraphics[width = 8cm]{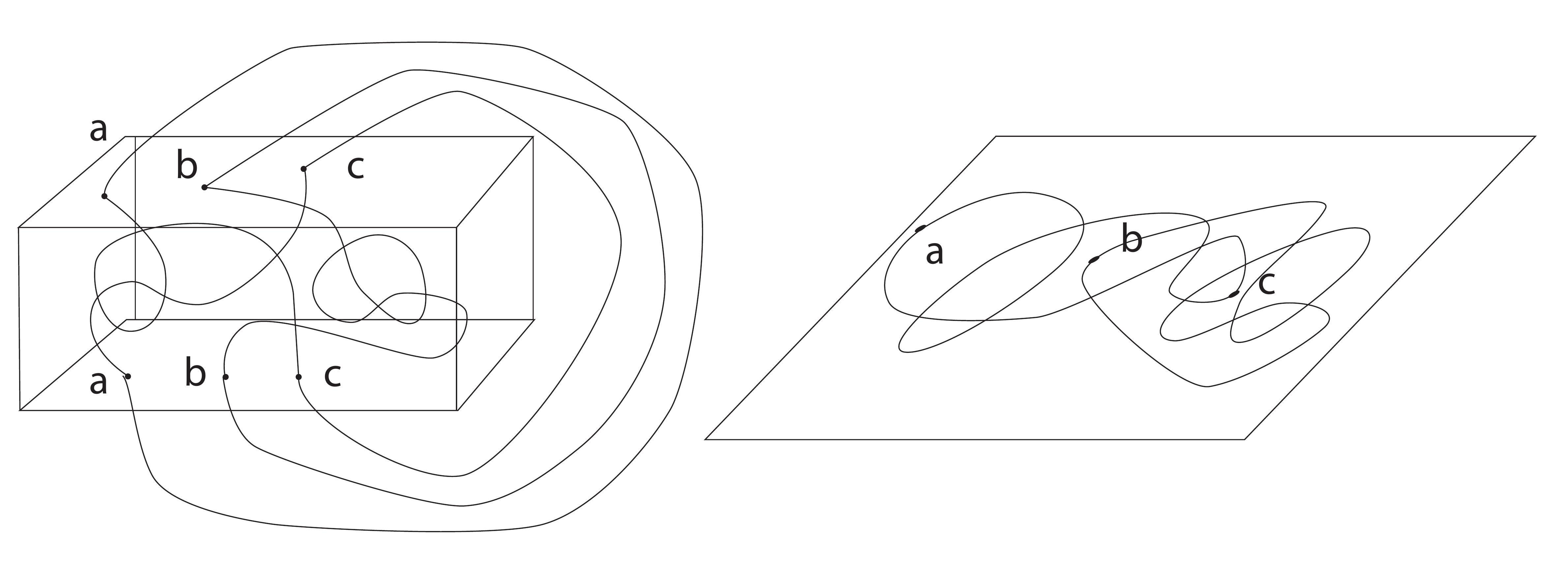}

\end{center}

\caption{}\label{fig:SS-diag-rel}
\end{figure}
Then there exists a natural diffeomorphism $f$ from $(S_{g} \times S^{1} )\backslash$  $(S_{g} \times \{x_{0}\})$ to $S_{g} \times (0,1) \subset S_{g} \times [0,1]$. Let $M_{L}$ $=$ $\overline{f((S_{g} \times S^{1}) - (S_{g} \times \{x_{0}\}))}$. Then $f(L)$ in $M_{L}$ consists of finitely many circles and arcs with exactly two boundaries on $S_{g} \times \{0\}$ and $S_{g} \times \{1\}$. Let $D_{f(L)}$ be the image of a projection of $f(L)$ on the plane. The diagram $D_{f(L)}$ of $L$ on $S_{g}$ has $n$-arcs with vertices and $m$-circles as described in the right of Fig. \ref{fig:SS-diag-rel}. Note that two arcs near to a vertex correspond to arcs near $S_{g} \times \{0\}$ and $S_{g} \times \{1\}$, respectively. We change a vertex to two small lines such that if one of the lines corresponds to an arc which is near to $S_{g} \times \{1\}$, the line is longer than another, as describe in Fig.~\ref{Vertex}.

\begin{figure}[h!]
\begin{center}
 \includegraphics[width = 8cm]{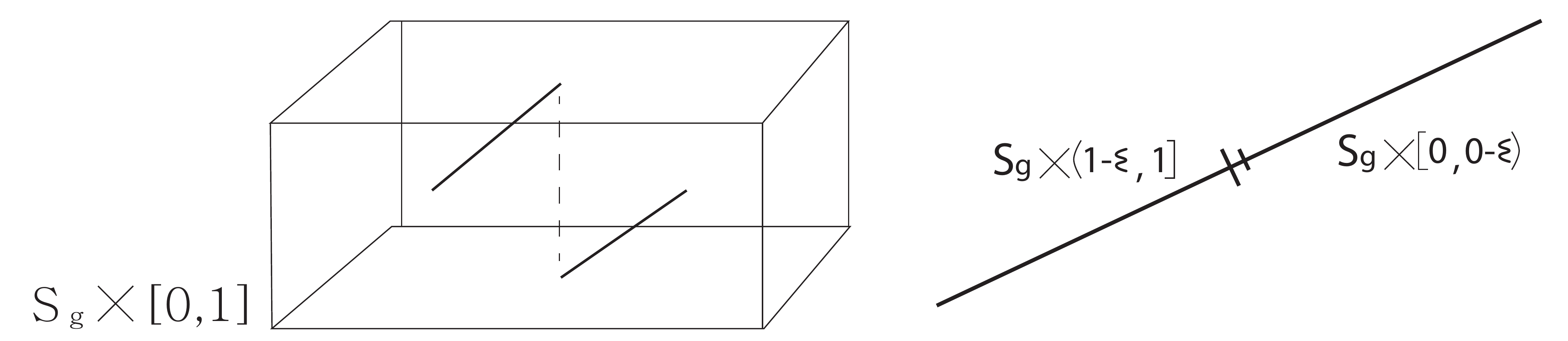}

\end{center}
\caption{}\label{Vertex}
\end{figure}

Since $D_{f(L)}$ is a framed 4-valent graph with double lines on the plane, which comes from $f(L)$ in $S_{g} \times [0,1]$, we can give classical and virtual crossings for each intersections. That is, a link $L$ in $M_{L}$ has a virtual knot diagram with double lines on the plane. The following theorem also holds.
\begin{thm}[M. K. Dabkowski, M. Mroczkowski (2009) \cite{DabkowskiMroczkowski}, Kim (2018) \cite{Kim}]\label{thm:diag_on_plane}
   Let $L$ and $L'$ be two links in $S_{g} \times S^{1}$. Let $D_{L}$ and $D_{L'}$ be diagrams of $L$ and $L'$ on the plane, respectively. Then $L$ and $L'$ are equivalent if and only if $D_{L'}$ can be obtained from $D_{L}$ by applying the following moves in Fig.~\ref{moves2}.

  \begin{figure}[h!]
\begin{center}
 \includegraphics[width = 12cm]{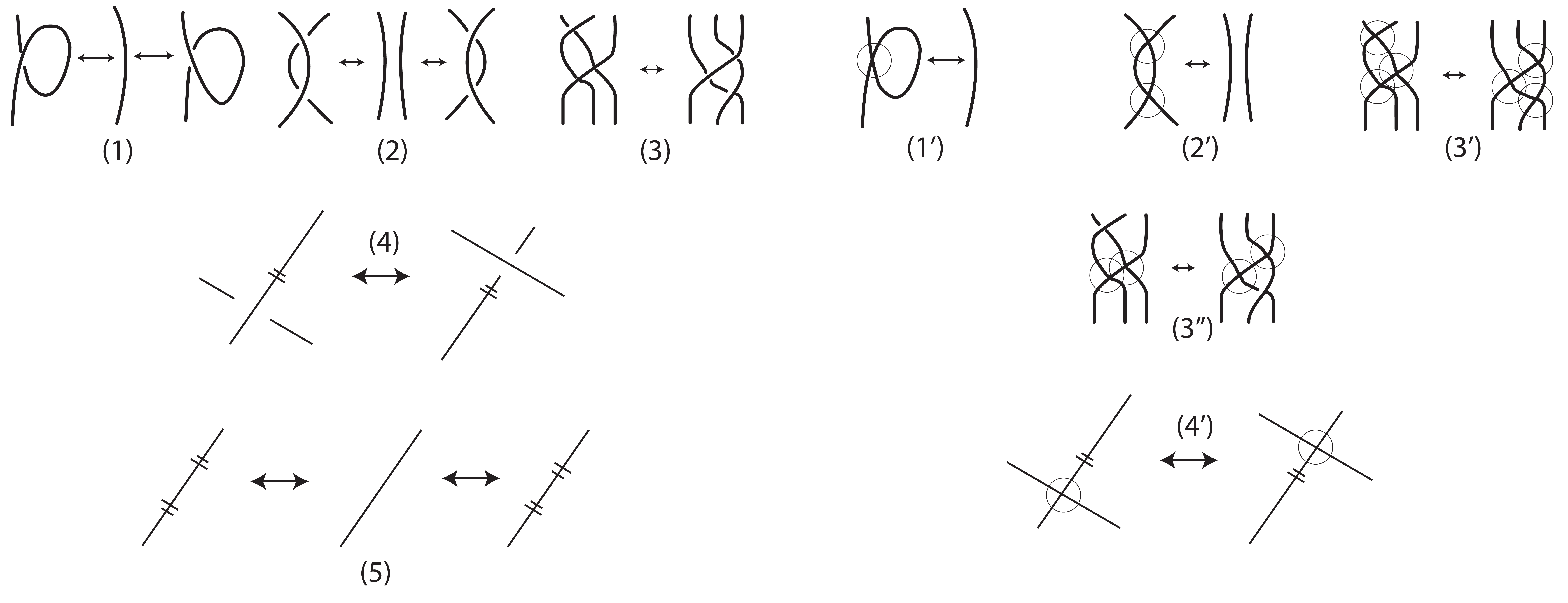}

\end{center}

 \caption{Moves for links in $S_{g}\times S^{1}$}\label{moves2}
\end{figure}
\end{thm}

\begin{rem}
The moves (4) and (5) correspond to local deformations of the link $L$ in $S_{g} \times S^{1}$ described in Fig.~\ref{fig:localmove4} and \ref{fig:localmove5}.
 \begin{figure}[h!]
\begin{center}
 \includegraphics[width = 6cm]{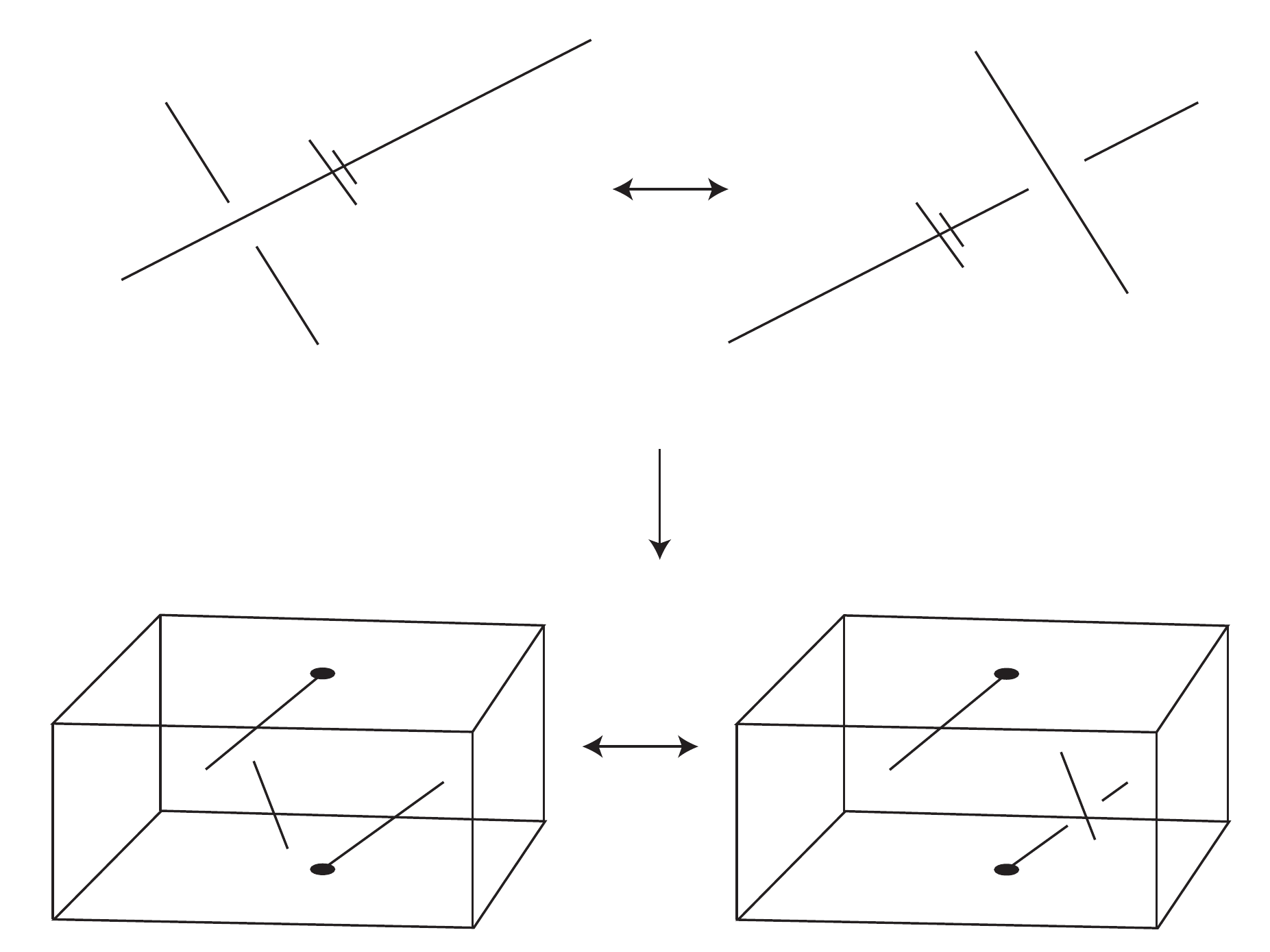}

\end{center}

\caption{Local image of a link corresponding to move 4}\label{fig:localmove4}
\end{figure}
 \begin{figure}[h!]
\begin{center}
 \includegraphics[width = 6cm]{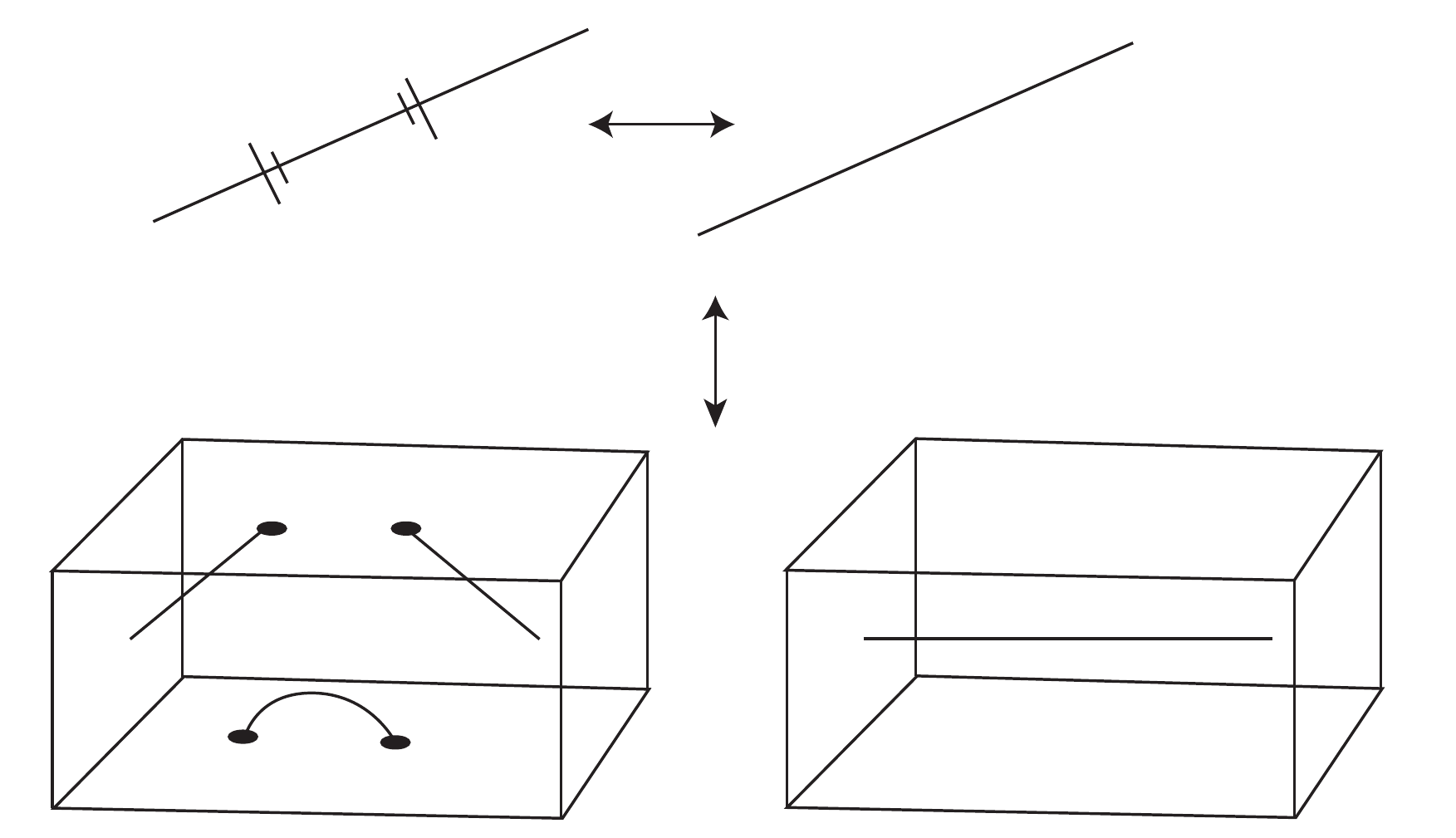}

\end{center}

\caption{Local image of a link corresponding to move 5}\label{fig:localmove5}
\end{figure}
\end{rem}
On the base of Theorem~\ref{thm:diag_on_plane} we can study knots by means of diagrams modulo local moves.

\begin{rem}
As described in Fig. \ref{cro_change}, by adding two double lines one can change over/under information of a crossing. 

  \begin{figure}[h!]
\begin{center}
 \includegraphics[width = 8cm]{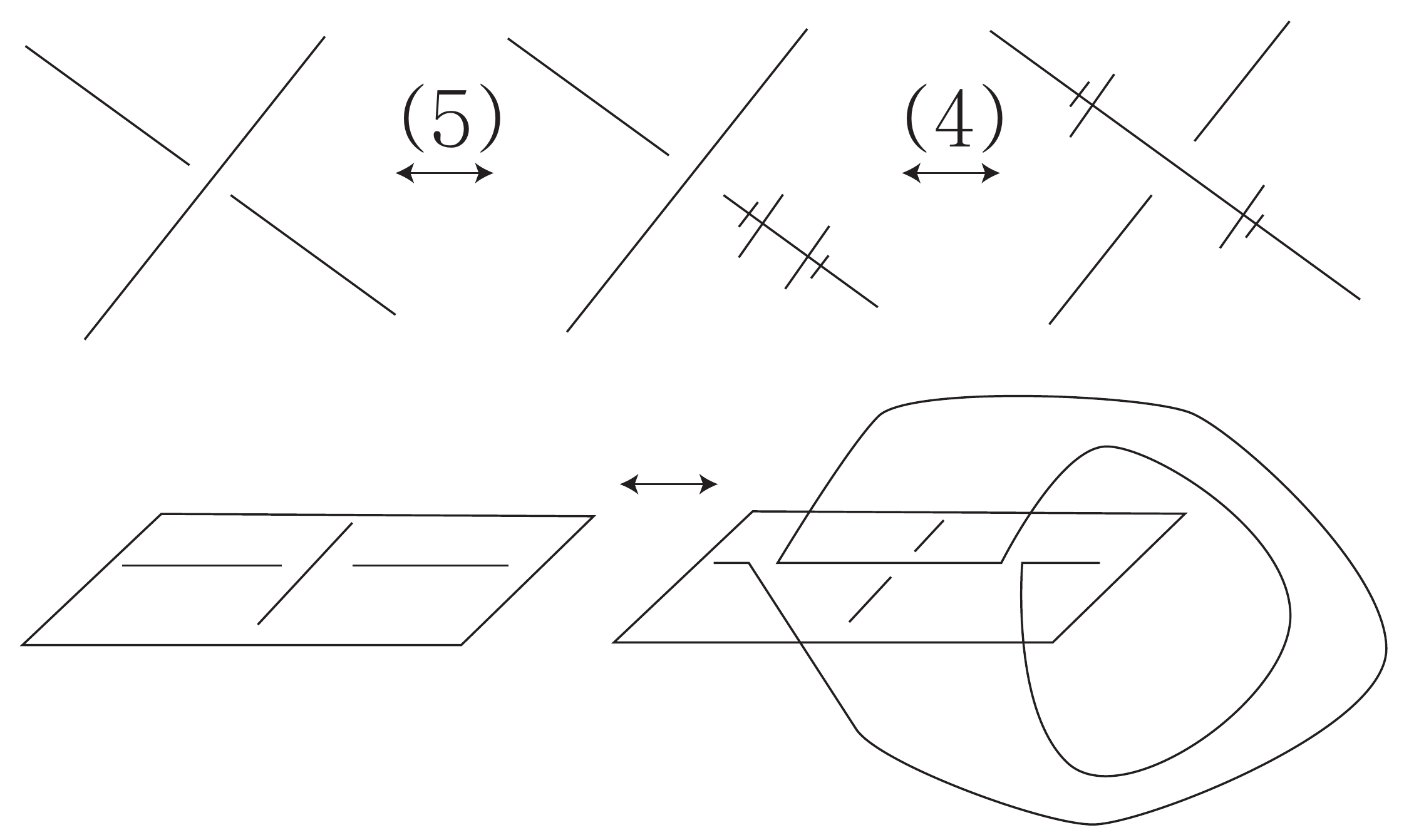}

\end{center}

 \caption{A crossing change with additional two double lines}\label{cro_change}
\end{figure}
\end{rem}
\begin{cor}
The moves in Theorem~\ref{thm:diag_on_plane} can be reformutaed by replacing the move 4 in Fig. \ref{moves2} by the move 4' as in Fig. \ref{moves3}.
 \begin{figure}[h!]
\begin{center}
 \includegraphics[width = 12cm]{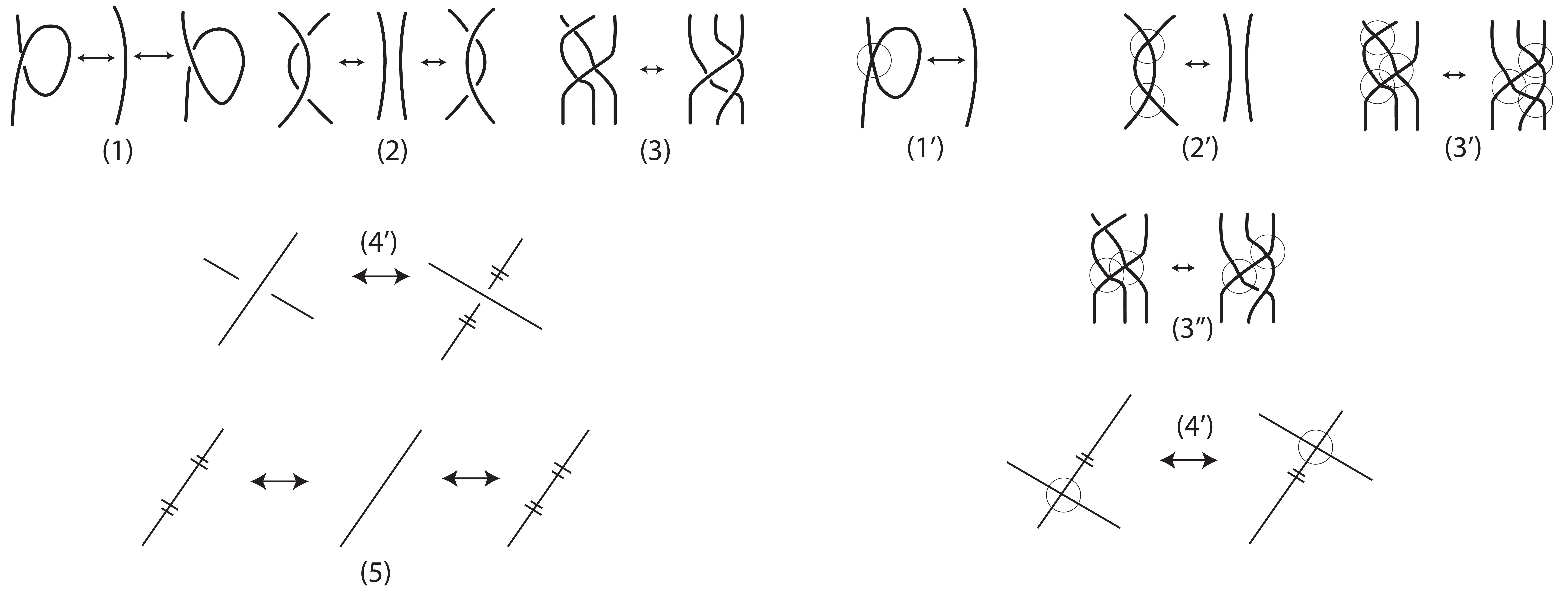}

\end{center}
 \caption{Moves for links in $S_{g} \times S^{1}$ obtained by replacing move 4 by move 4'}\label{moves3}
\end{figure}
\end{cor}

\begin{rem}
From the previous remark, one can say that we lose over/under information for classical crossings which comes from the fiber $[0,1]$ on $S_{g}$. But, instead of it, we can obtain an information from the fiber $S^{1}$ on $S_{g}$, which will be described as ``how many times a half of a crossing turns around''.
\end{rem}
\subsection{Degree of knots in $S_{g}\times S^{1}$} From now on we are mainly interested in {\it knots} in $S_{g} \times S^{1}$. 
The most important information from knots in $S_{g}\times S^{1}$ is ``how many times the knot rotates along $S^{1}$''. More precisely, we consider the natural covering $\Pi: \mathbb{R} \rightarrow S^{1}$ defined by $\Pi(r) = e^{2\pi r i}$. Then the function $Id_{S_{g}} \times \Pi : S_{g}\times \mathbb{R} \rightarrow S_{g}\times S^{1}$ is also a covering over $S_{g}\times S^{1}$ where $Id_{S_{g}} : S_{g} \rightarrow S_{g}$ is the identity map.\\

Let $K : [0,1] \rightarrow S_{g}\times S^{1}$ be a knot with $K(0)=K(1)$.
 Let $\tilde{K}$ be a lifting of $K$ into $S_{g} \times \mathbb{R}$ along a covering $Id_{S_{g}} \times \Pi : S_{g}\times \mathbb{R} \rightarrow S_{g}\times S^{1}$. When $\phi_{2} \circ \hat{K}(0) =0$, {\em the degree $deg(K)$ of a knot $K$ in $S_{g} \times S^{1}$} is defined by 
 $$deg(K) = \phi_{2} \circ \hat{K}(1).$$
 It is easy to see that the degree $deg(K)$ of a knot $K$ in $S_{g} \times S^{1}$ is an invariant.
 
   \begin{figure}[h!]
\begin{center}
 \includegraphics[width = 5cm]{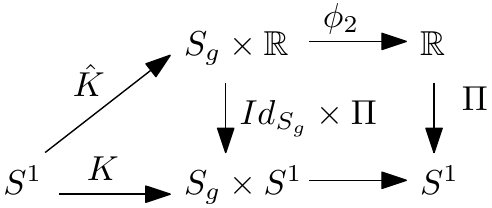}

\end{center}
\end{figure}
 
 For a line segment $l$ of a diagram $D$ of $K$, there is a line segment $l'$ in $\tilde{K}$ in $S_{g} \times \mathbb{R}$ corresponding to $l$. Let $\phi_{2} : S_{g} \times \mathbb{R} \rightarrow \mathbb{R}$ be a natural projection. If $\phi_{2}(l') \in [a, a+1)$, then we give a label $a$ to $l$. We consider the label $a$ as an element of $\mathbb{Z}$.


\begin{rem}\label{remark}
  For labels $a,b$ in the following figure, $b=a+1$.
  \begin{figure}[h!]
\begin{center}
 \includegraphics[width = 3cm]{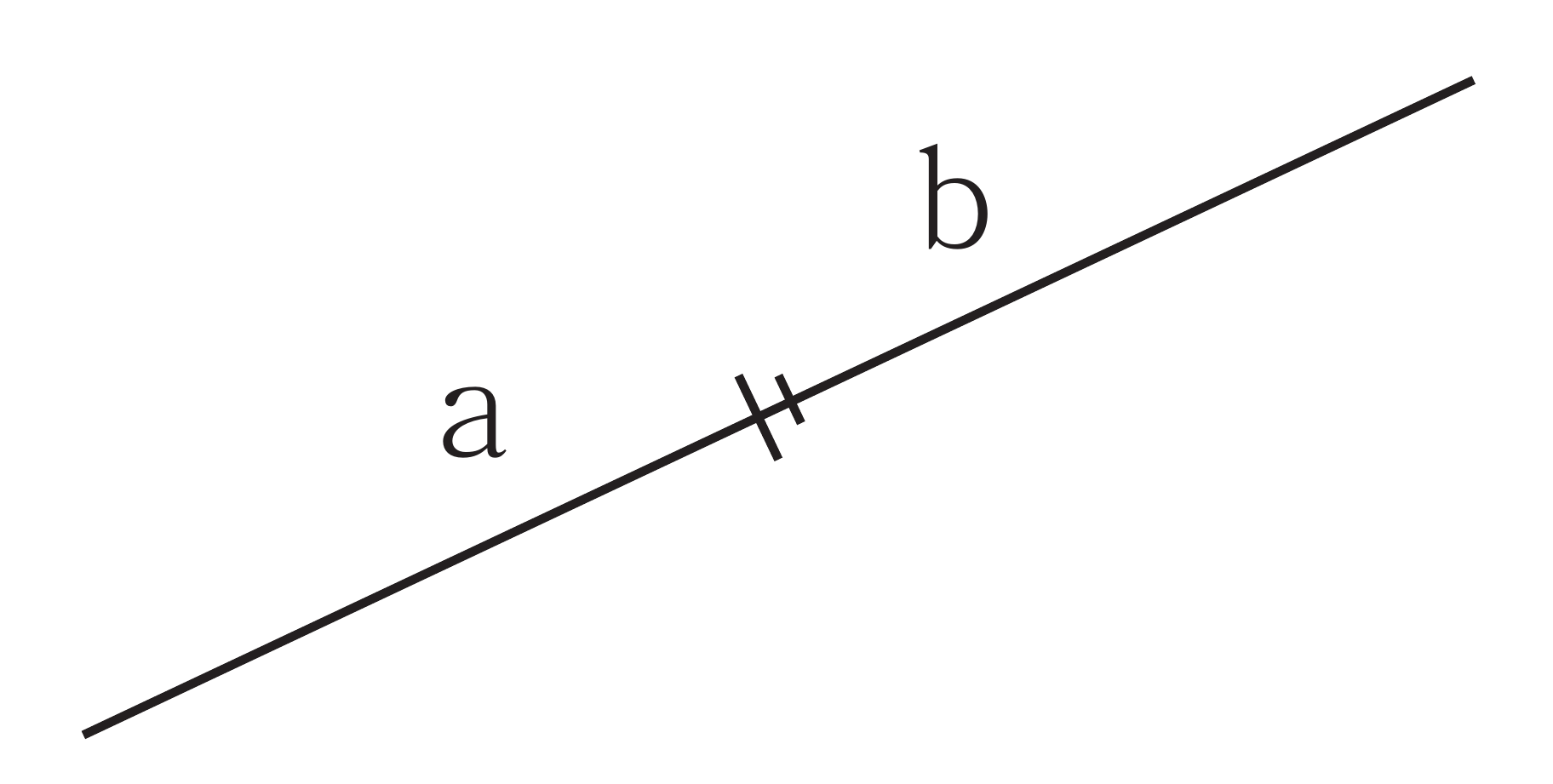}

\end{center}
 \caption{Heights of arcs near to double lines}
\end{figure}
\end{rem}

\begin{exa}
Let $K$ be a knot in $D^{2} \times S^{1}$ as described in Fig.~\ref{exa_label_arc}, where $D^{2}$ is a 2-dimensional disc. The knot $K$ has the degree $2$ and it has a diagram $D_{K}$ of trivial knot with two double lines. One can see that the arc of $K$ colored by red corresponds to the arc of $\tilde{K}$ placed in $D^{2} \times (0,1)$, but the arc of $K$ colored by green corresponds to the arc of $\tilde{K}$ placed in $D^{2} \times (1,2)$. Note that the red and green arcs of $K$ correspond to the arcs of $D_{K}$ colored by red and green respectively. Now we give numbers $0$ and $1$ to red and green arcs of $D_{K}$ respectively. Note that here the numbers $0$ and $1$ are considered as elements in $\mathbb{Z}_{2}$. 
  \begin{figure}[h!]
\begin{center}
 \includegraphics[width =8cm]{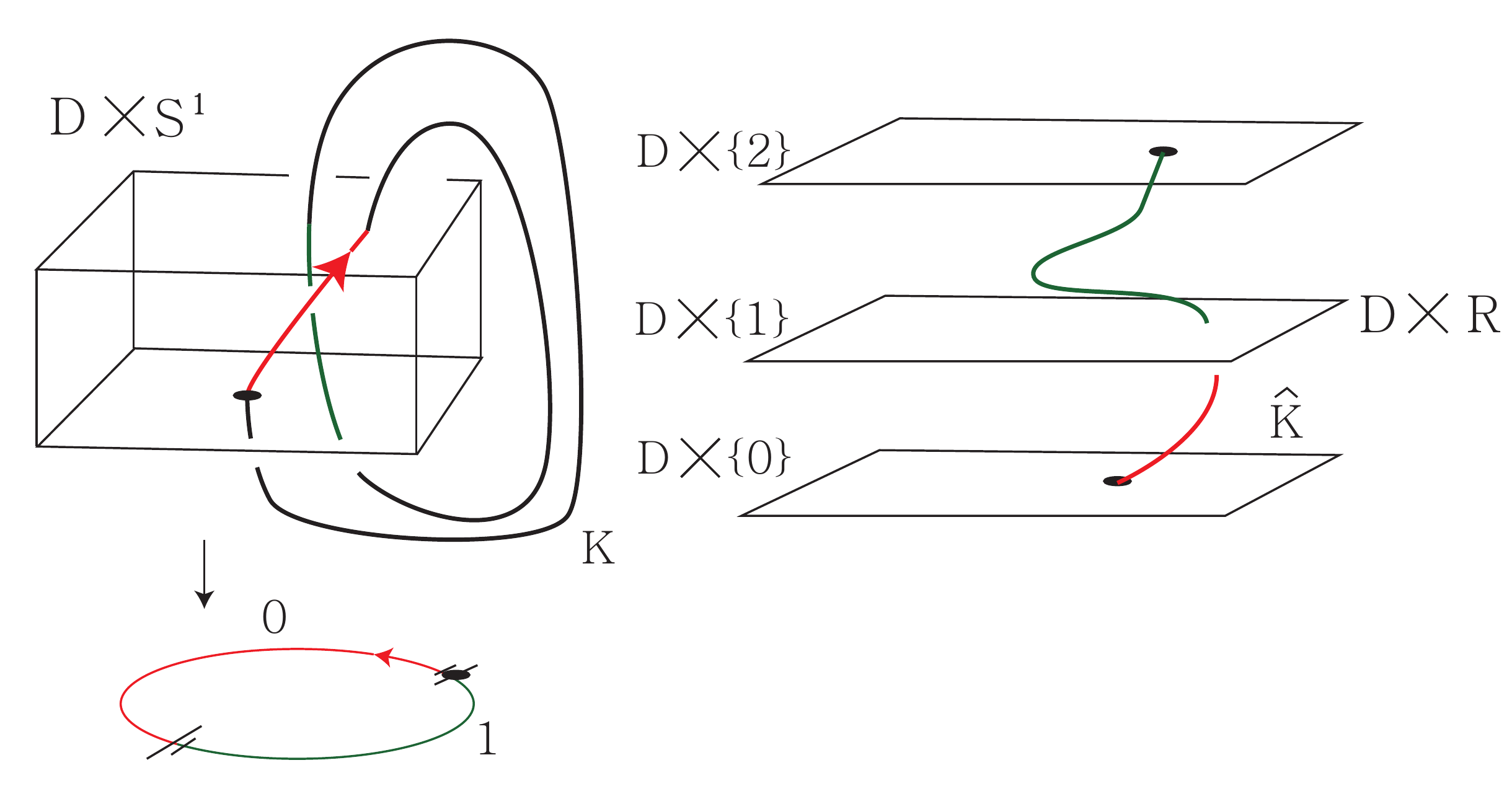}

\end{center}
 \caption{A knot in $D^{2} \times S^{1}$ with degree $2$}\label{exa_label_arc}
\end{figure}
\end{exa}

For each crossing, if over-arc is labeled by a number $b$ and under-arc is labeled by $a$ for some $a,b \in \mathbb{Z}$, then give a label $i = b - a$ to the crossing where $b-a$ is in $\mathbb{Z}$. Then we call $D$ with such labeling for each classical crossing a \textit{labeled diagram}.
\begin{figure}[h!]
\begin{center}
 \includegraphics[width = 3cm]{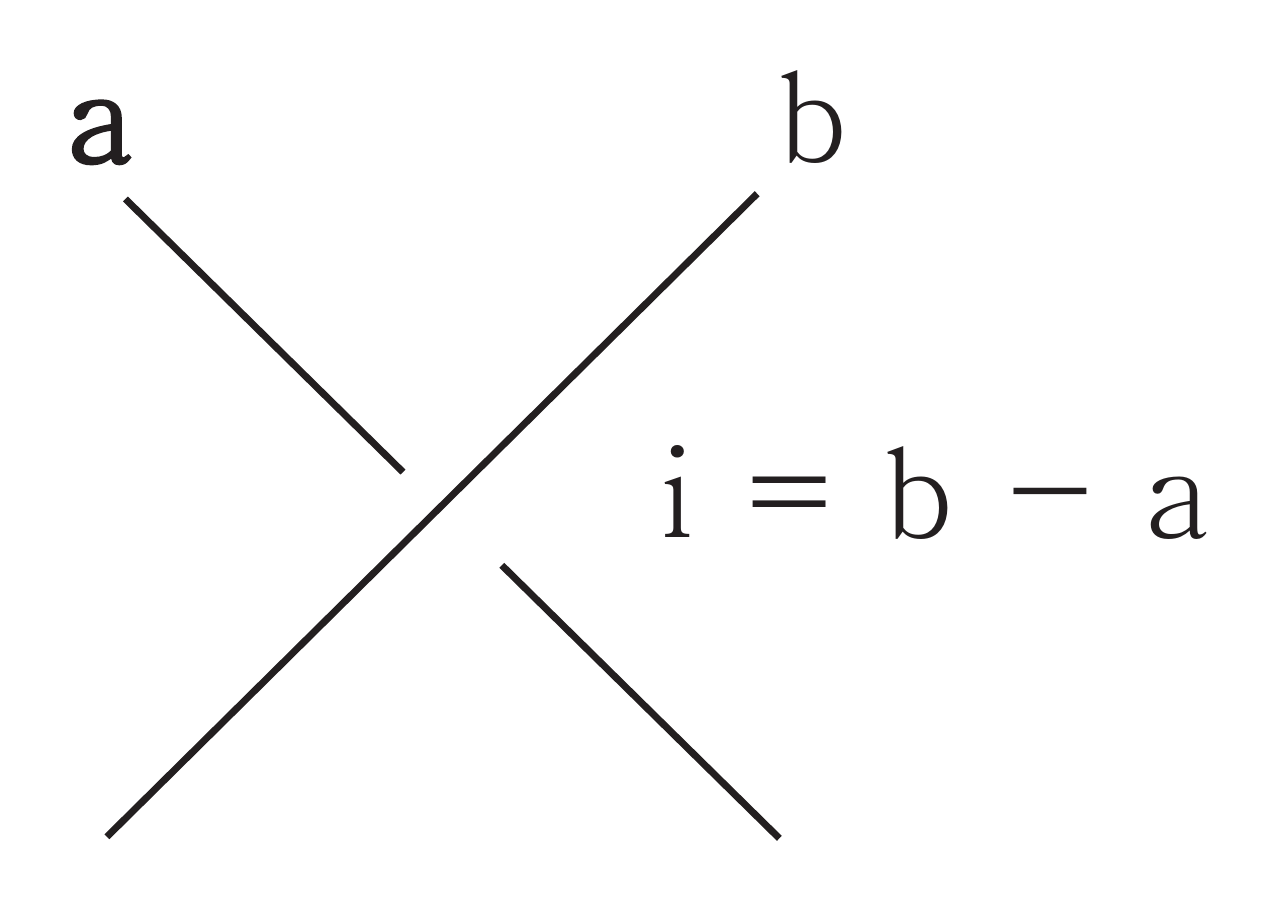}

\end{center}
 \caption{Label of crossing with heights $a$ and $b$ of under crossing and over crossing of a crossing}
\end{figure}
\begin{lem}[\cite{KimSgS1}]
 The labels for crossings satisfy the properties described in Fig.~\ref{labelmoves2}.\label{lem:labelmoves}

  \begin{figure}[h!]
\begin{center}
 \includegraphics[width = 10cm]{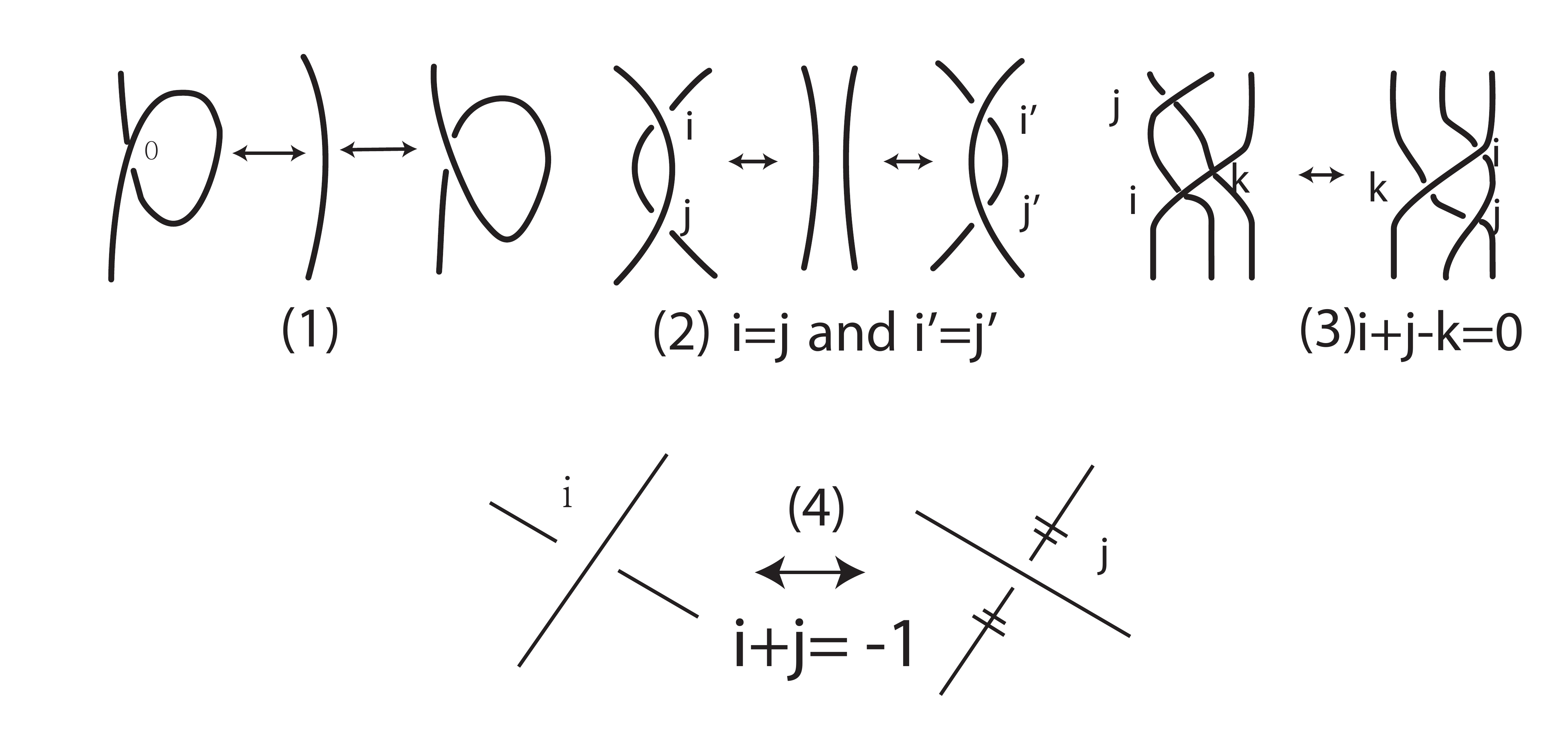}

\end{center}
 \caption{Properties of labels}\label{labelmoves2}
\end{figure}
\end{lem}

\begin{rem}
Geometrically, the label of a crossing $c$ means how many times the curve from the crossing $c$ to itself turns around $S^{1}$.
\end{rem}

\begin{rem}

If we consider the labels for classical crossing modulo $2$, then it becomes the parity. If we separate classical crossings by label $0$ and others, then it is a weak parity.
\end{rem}

\section{Winding parity for knots in $S_{g} \times S^{1}$}
In the present section, let us extend the ``labels'' axiomatically.
\begin{dfn}
Let $A$ be an abelian group. A {\em winding parity on diagrams of a knot $\mathcal{K}$ with coefficients in $A$} is a family of maps $p_{K} : \mathcal{V}(K) \rightarrow A$, $K \in ob(\mathcal{K})$, such that for any elementary morphism $f:K \rightarrow K'$ conditions described in Fig.~\ref{fig:label-condition} hold:

  \begin{figure}[h!]
\begin{center}
 \includegraphics[width = 10cm]{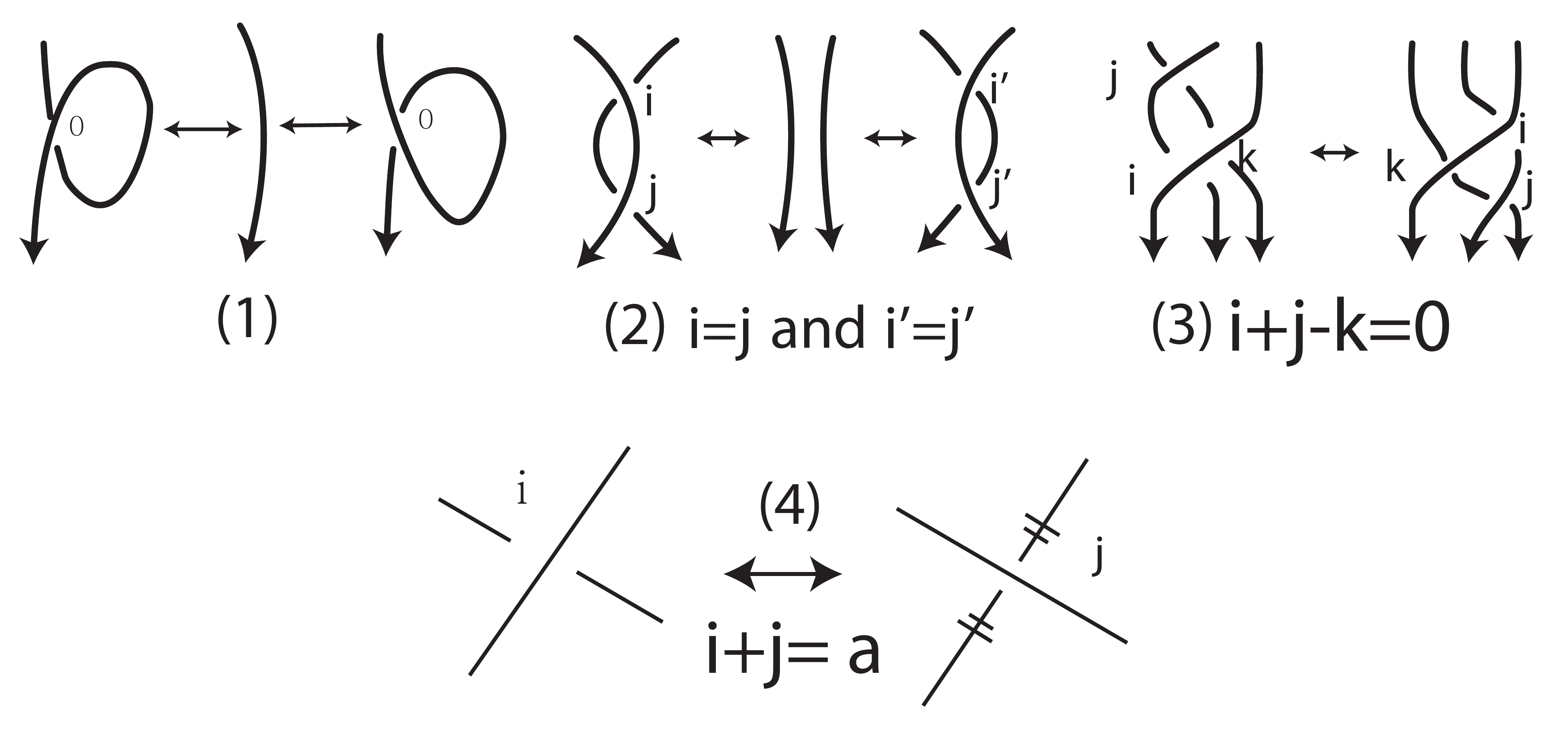}

\end{center}
 \caption{Winding parity condition}\label{fig:label-condition}
\end{figure}
\end{dfn}

\begin{exa}
 The label, which is defined in the previous section, is a winding parity with an abelian group $\mathbb{Z}$ (or $\mathbb{Z}_{n}$) and fixed element $a=-1$.
 \end{exa}
 
 \begin{exa}
 The Gaussian parity is a winding parity with an abelian group $\mathbb{Z}_{2}$ and the fixed element $a=0$.
 \end{exa}
 
 \begin{dfn}[\cite{Nikonov-weakp}]
 {\em An oriented parity $p$} is a family of maps $p_{D}: \mathcal{V}(D) \rightarrow G$ defined for each diagram $D$ of the knot $\mathcal{K}$, that possesses the properties described in Fig.~\ref{fig:label-oriented}.
   \begin{figure}[h!]
\begin{center}
 \includegraphics[width = 8cm]{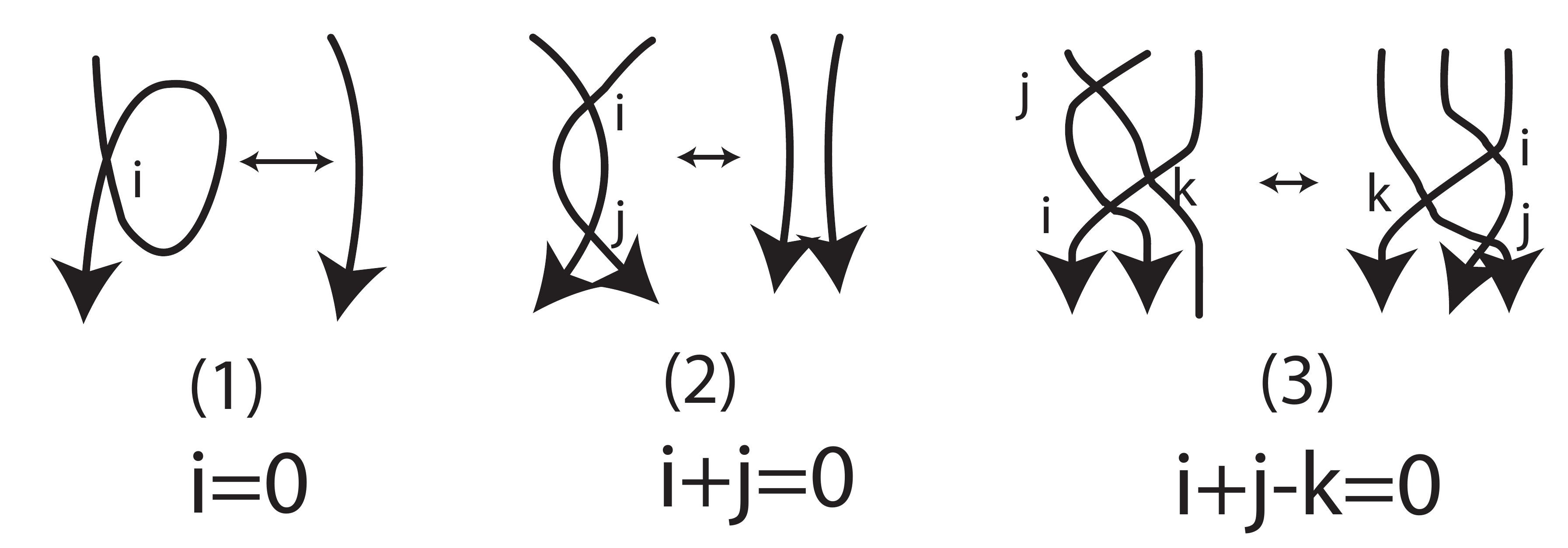}

\end{center}
 \caption{Properties of the oriented parity}\label{fig:label-oriented}
\end{figure}
 \end{dfn}
  \begin{rem}
 Let $\{p_{D}\}$ be a winding parity with $a=0$ in $A$. Then $\{ sgn(v)p_{D}\}$ is an oriented parity. 
 \end{rem}
 
 \subsection{Homological winding parity for crossings of $S_{g}\times S^{1}$}
 
 Let us define labels for crossings of knots in $S_{g}\times S^{1}$ as follows:

Let $c$ be a crossing of a diagram with double lines of a knot in $S_{g} \times S^{1}$. Let us consider its local pre-images in $S_{g} \times [0,1]$ and we connect them by straight lines $v_{c}$ oriented from over crossing to under crossingas described in Fig.~\ref{fig:oriented-half-crossing}. By connecting the straight line with the arc beginning from $c$ to $c$, we obtain a half $\gamma_{c}$. Let us define {\em the homological winding parity of a crossing $c$} by $[\gamma_{c}] \in H_{1}(S_{g}\times S^{1})/[K]$, where $[K]$ is the equivalence class of $K$ in $H_{1}(S_{g}\times S^{1})/[K]$.

 \begin{figure}[h!]
\begin{center}
 \includegraphics[width = 4.5cm]{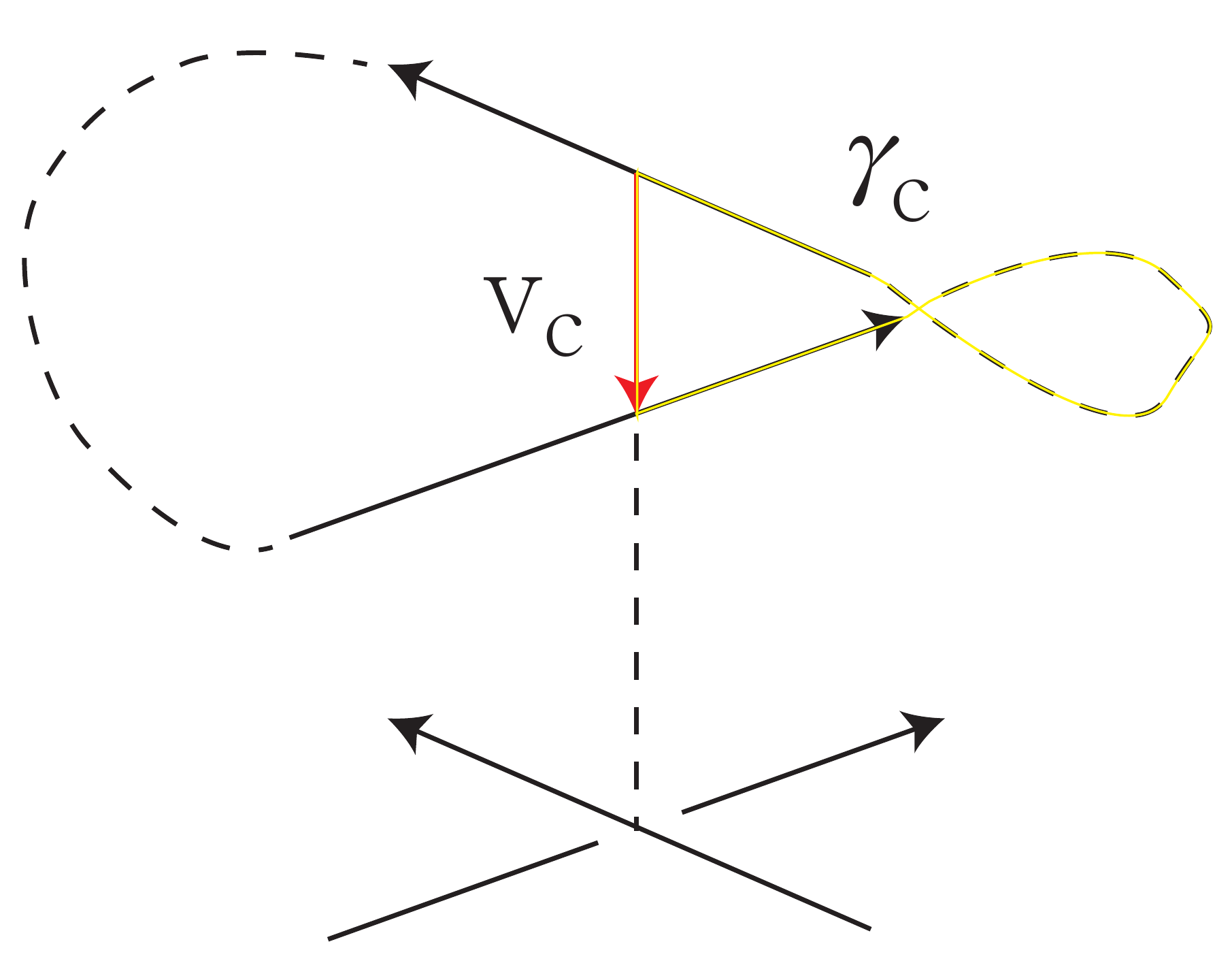}
\end{center}
\caption{A half $\gamma_{c}$ of a crossing $c$}\label{fig:oriented-half-crossing}
\end{figure}

\begin{thm}
The homological winding parity is a winding parity with the abelian group $H_{1}(S_{g} \times S^{1})/[K]$ and the fixed element $a= -[\{*\} \times S^{1}]$.

\end{thm}

\begin{proof}
For move 1, let us show that, when a positive (see Fig.\ref{fig:proof-rotation-move1-1}) or negative (see Fig.\ref{fig:proof-rotation-move1-2}) kink disappears, the corresponding crossing $c$ has a homological parity $[\gamma_{c}]=0$ in $H_{1}(S_{g} \times S^{1})/[K]$. For a positive kink, the curve $\gamma_{c}$ is exactly the kink with vertical line. In this case, since $\gamma_{c}$ is contractible, one can see that $[\gamma_{c}]=0$ in $H_{1}(S_{g} \times S^{1})/[K]$. 
\begin{figure}[h!]
\begin{center}
 \includegraphics[width = 5cm]{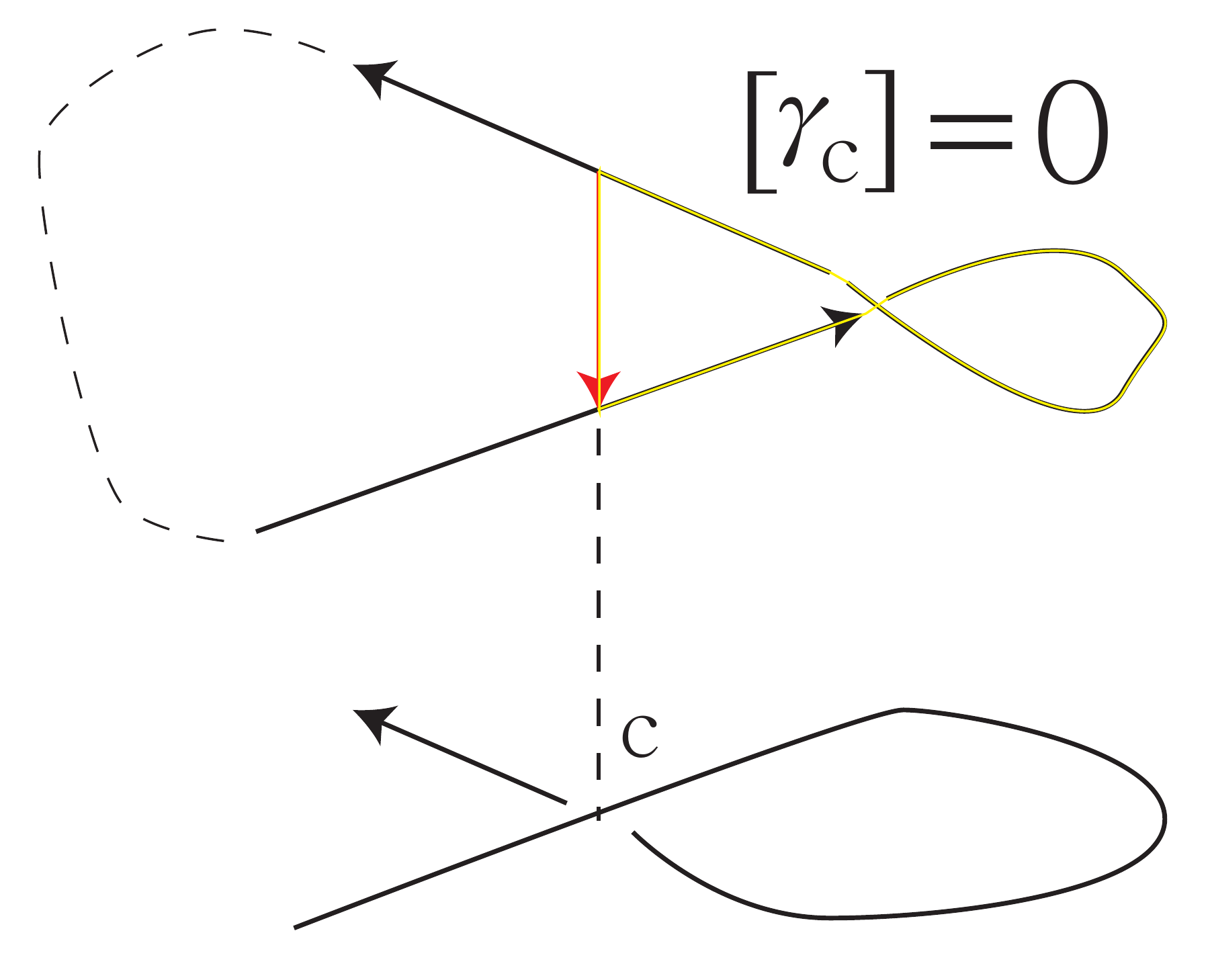}
\end{center}
\caption{Move 1: A positive kink}\label{fig:proof-rotation-move1-1}
\end{figure}
For a negative kink, let us consider a closed curve $\gamma_{c}'$ such that $\gamma_{c}'\cdot\gamma_{c} = K$. Then $\gamma_{c}'$ contains the kink with vertical link and $[\gamma_{c}']=0$ in $H_{1}(S_{g} \times S^{1})/[K]$. Since $[K] = [\gamma_{c}'\cdot\gamma_{c}] = [\gamma_{c}']+[\gamma_{c}] =0$ in $H_{1}(S_{g} \times S^{1})/[K]$, one can see that $[\gamma_{c}]=0$.
\begin{figure}[h!]
\begin{center}
 \includegraphics[width = 5cm]{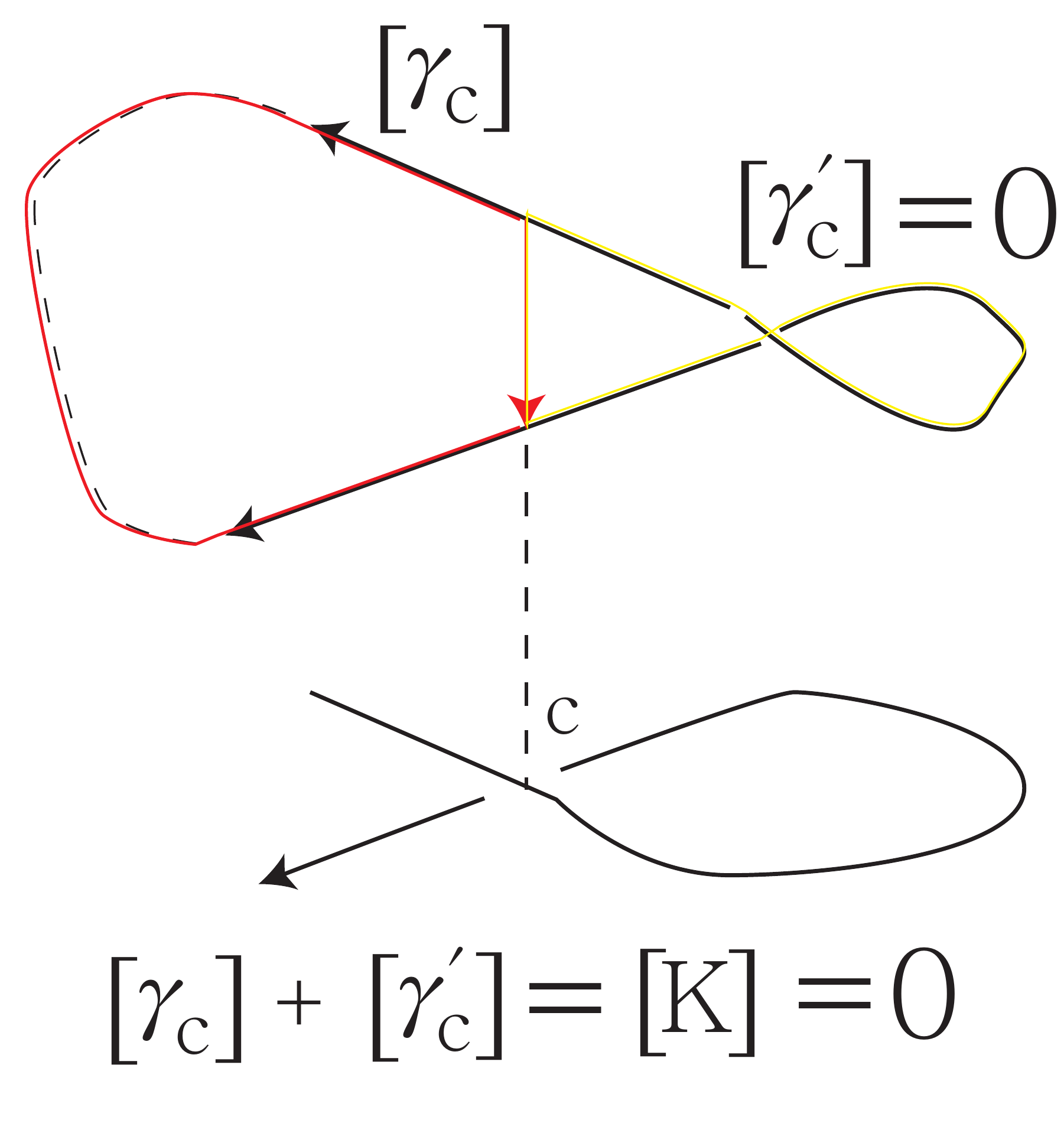}
\end{center}
\caption{Move 1: A negative kink}\label{fig:proof-rotation-move1-2}
\end{figure}

For move 2, we have two cases as described in Fig.\ref{fig:proof-rotation-move2-1} and Fig.\ref{fig:proof-rotation-move2-2}. In Fig.\ref{fig:proof-rotation-move2-1} it is obvious that there is a disc with boundary $C=u_{cc'}\cup v_{c'} \cup -u_{cc'} \cup -v_{c}$ and  $C*\gamma_{c}$ and $\gamma_{c'}$ are homotopic. Therefore $[\gamma_{c}]=[\gamma_{c'}]$ in $H_{1}(S_{g} \times S^{1})/[K]$.
\begin{figure}[h!]
\begin{center}
 \includegraphics[width = 7cm]{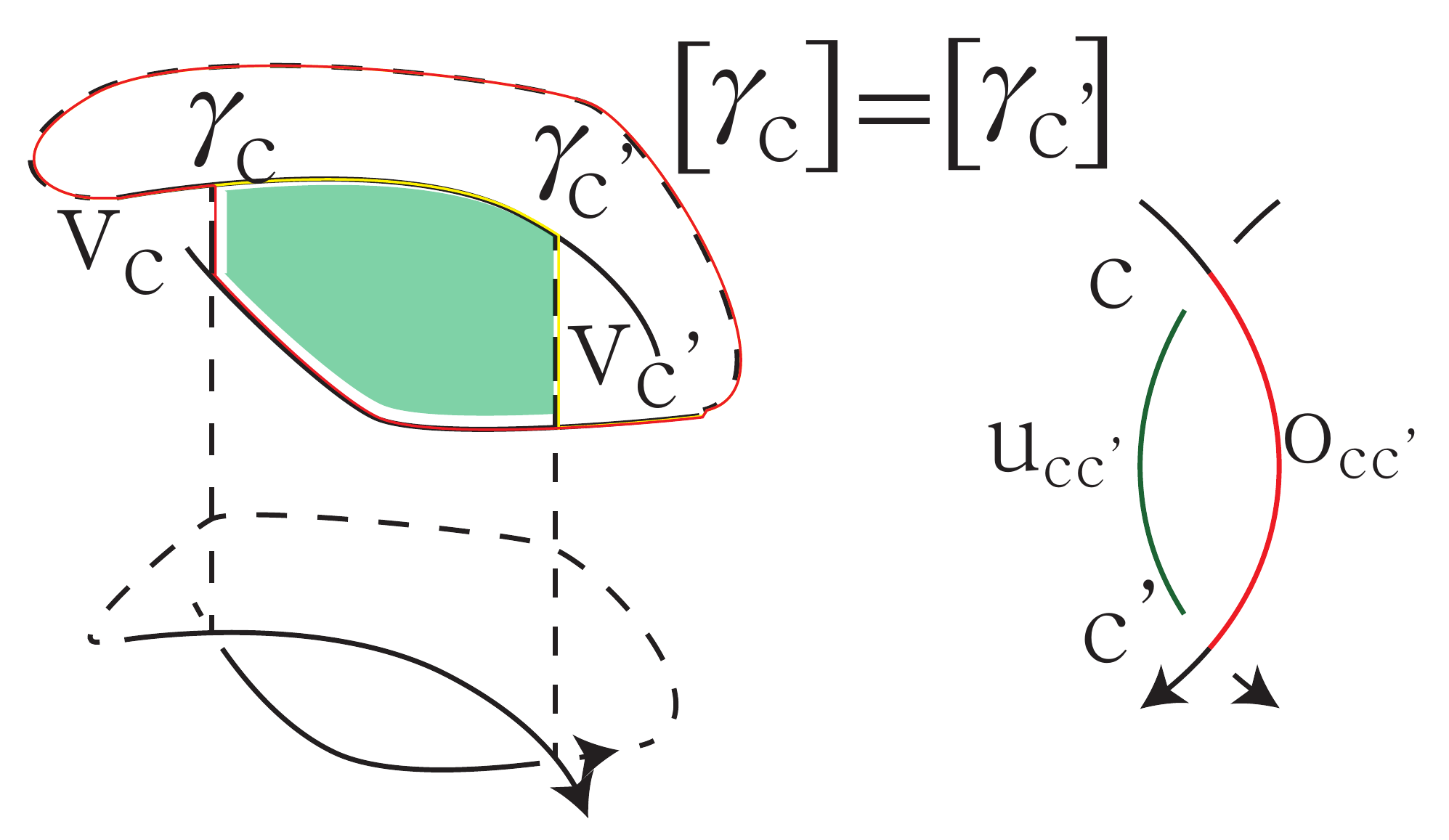}
\end{center}
\caption{Move 2, Case 1}\label{fig:proof-rotation-move2-1}
\end{figure}
Analogously in the case described in Fig.\ref{fig:proof-rotation-move2-2} one can show that $[\gamma_{c}]=[\gamma_{c'}]$ in $H_{1}(S_{g} \times S^{1})/[K]$.
\begin{figure}[h!]
\begin{center}
 \includegraphics[width = 7cm]{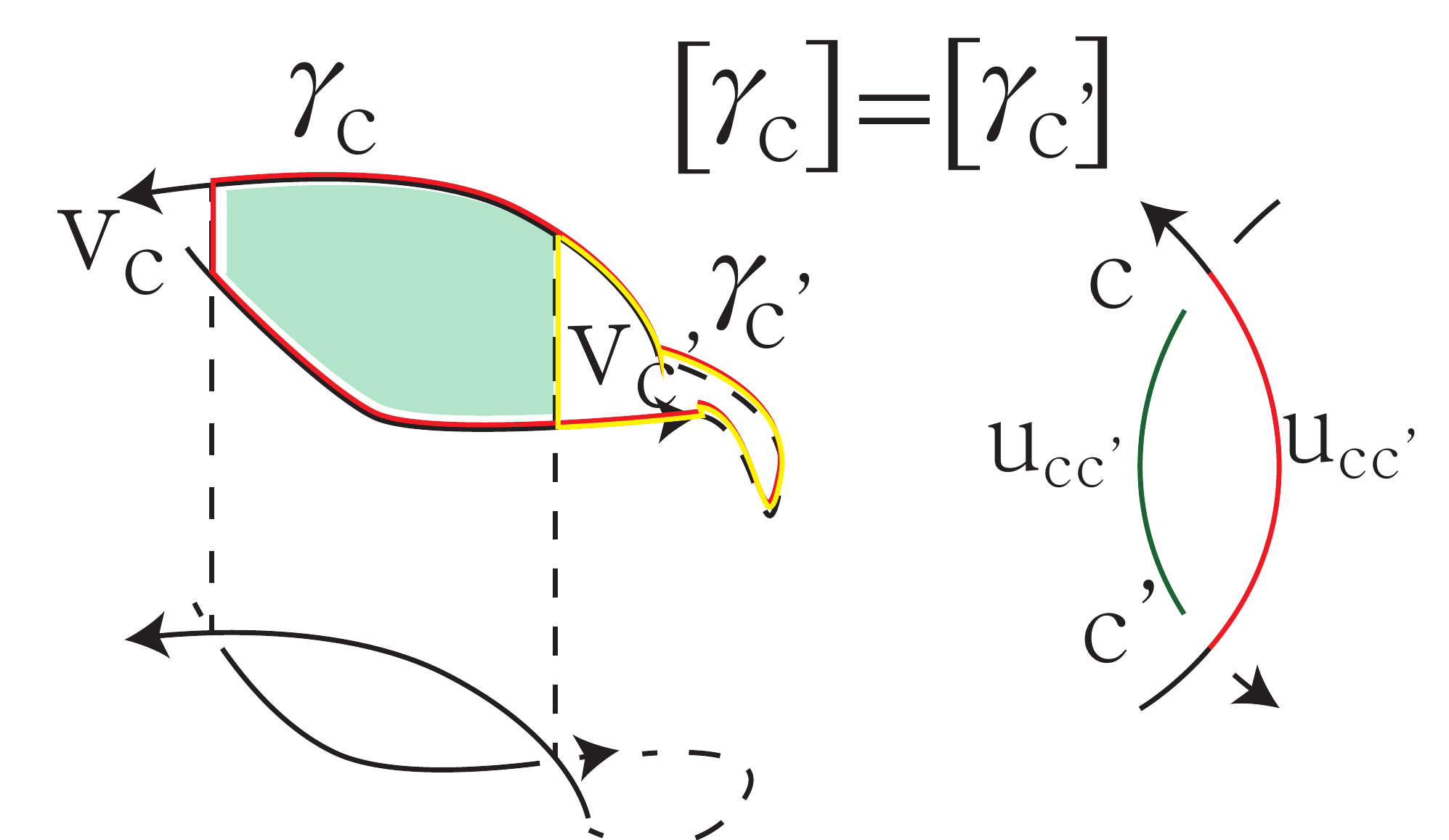}
\end{center}
\caption{Move 2, Case 2}\label{fig:proof-rotation-move2-2}
\end{figure}

For move 3, see Fig.\ref{fig:proof-rotation-move3-1} and \ref{fig:proof-rotation-move3-2}. We schemetically describe $\gamma$'s for crossings $a,b,c,a',b',c'$ in the move 3. First, since two closed curve $u_{bc} * v_{c}* (-u_{ac})*v_{a}*u_{ab} * (-v_{b})$ and $u_{c'b'}*v_{b'}*u_{b'a'}*(-v_{a'})*(-u_{c'a'})*(-v_{c'})$ are contractible, analogously to the proof for move 2, one can show that $[\gamma_{a}]= [\gamma_{a'}]$, $[\gamma_{b}]= [\gamma_{b'}]$ and $[\gamma_{c}]= [\gamma_{c'}]$ in both cases. 

If ends of strings in move 3 connected as described in Fig.\ref{fig:proof-rotation-move3-1},  

\begin{figure}[h!]
\begin{center}
 \includegraphics[width = 9cm]{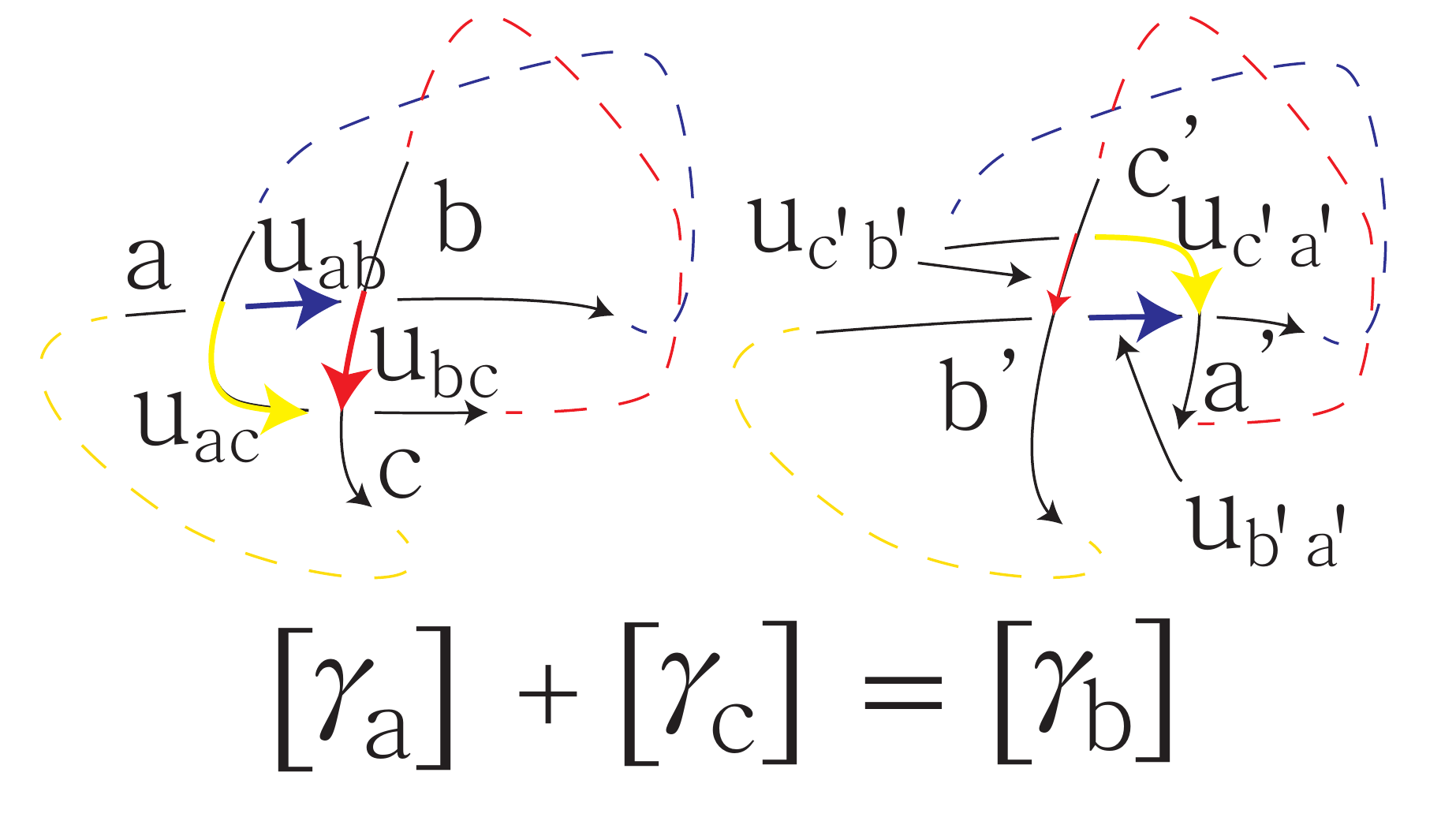}
\end{center}
\caption{Move 3, Case 1}\label{fig:proof-rotation-move3-1}
\end{figure}

then one can obtain that
\begin{eqnarray*}
\gamma_{b} &\sim& (-u_{ab})*\gamma_{a}* u_{ab}* (-v_{b} *u_{bc} * v_{c})*\gamma_{c} * (-v_{c} *(-u_{bc})*v_{b}) \\
&\sim& u_{ab}^{-1}*\gamma_{a} * u_{ab}*(-v_{b} *u_{bc} * v_{c})*\gamma_{c} *(-v_{b} *u_{bc} * v_{c})^{-1}.
\end{eqnarray*}
Therefore 
\begin{eqnarray*}
[\gamma_{b}] &=& [u_{ab}^{-1}*\gamma_{a} * u_{ab}*(-v_{b} *u_{bc} * v_{c})*\gamma_{c} *(-v_{b} *u_{bc} * v_{c})^{-1}] \\
&=& [u_{ab}^{-1}*\gamma_{a} * u_{ab}] + [(-v_{b} *u_{bc} * v_{c})*\gamma_{c} *(-v_{b} *u_{bc} * v_{c})^{-1}]\\
&=& [\gamma_{a}]+[\gamma_{c}].
\end{eqnarray*}

If ends of strings in move 3 connected as described in Fig.\ref{fig:proof-rotation-move3-2},  
\begin{figure}[h!]
\begin{center}
 \includegraphics[width = 9cm]{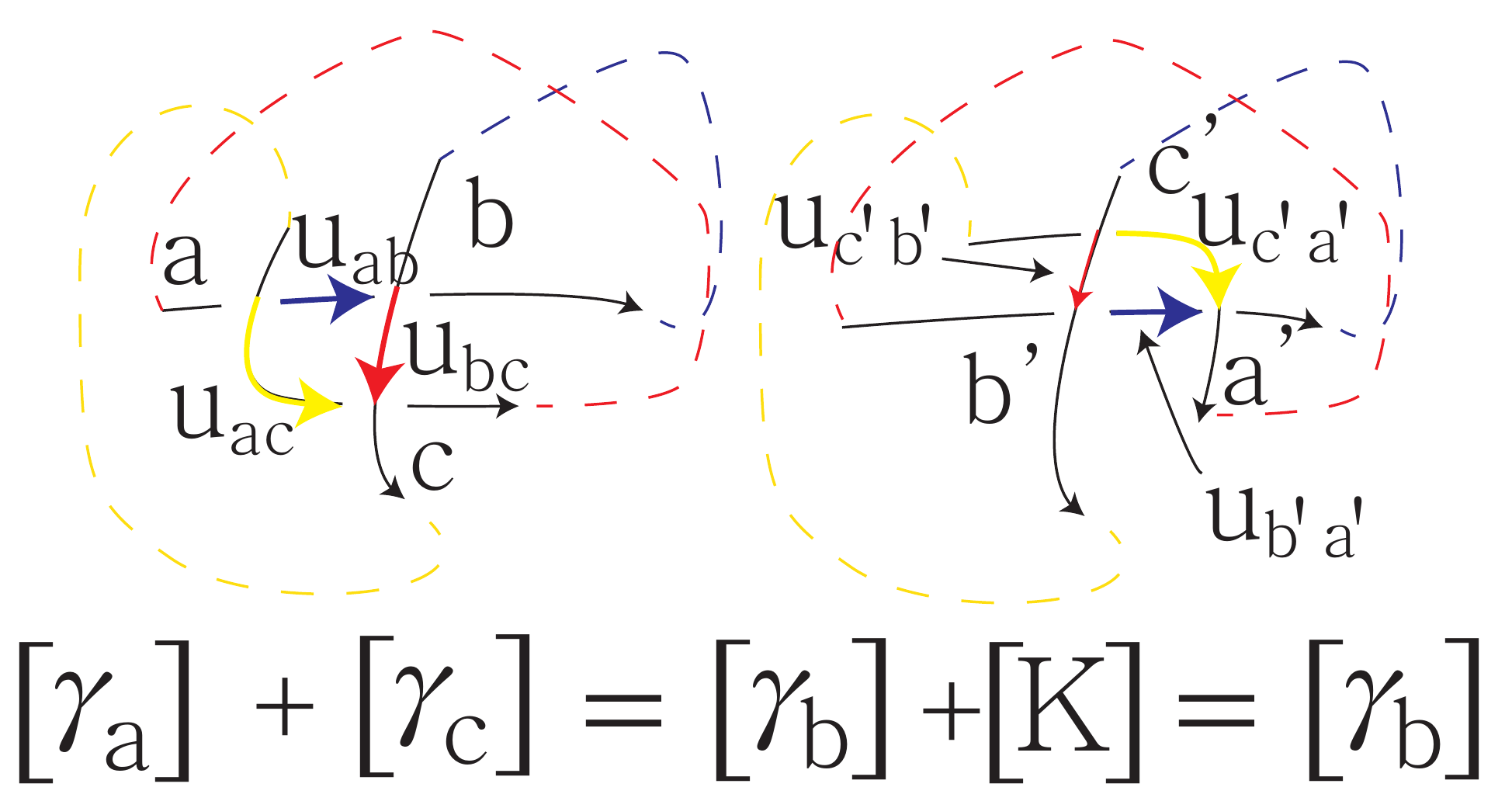}
\end{center}
\caption{Move 3, Case 2}\label{fig:proof-rotation-move3-2}
\end{figure}
analogously to the previous case we obtain $[\gamma_{a}]+[\gamma_{c}] = [K]+[\gamma_{b}] = [\gamma_{b}] $ in $H_{1}(S_{g}\times S^{1})/[K]$.

For move 4, see Fig.\ref{fig:proof-rotation-move4}. Note that $\gamma_{c}*v_{c}^{-1}* \gamma_{c'}* v_{c'}^{-1} \sim \gamma_{c}*v_{c}^{-1}* v_{c'}^{-1} * v_{c'} *\gamma_{c'}* v_{c'}^{-1} \sim K$ and 
$v_{c} * v_{c'} \sim \{*\} \times (-S^{1})$ where $*$ is an arbitrary point on $v_{c}$. Therefore we obtain that
\begin{eqnarray*}
0 = [K] &=& [ \gamma_{c}*v_{c}^{-1}* v_{c'}^{-1} * v_{c'} *\gamma_{c'}* v_{c'}^{-1}] \\
&=& [ \gamma_{c}]+ [v_{c}^{-1}* v_{c'}^{-1}] + [ v_{c'} *\gamma_{c'}* v_{c'}^{-1}] \\&=& [\gamma_{c}]+[\gamma_{c'}] - [v_{c}^{-1}* v_{c'}^{-1}]\\
&=& [\gamma_{c}]+[\gamma_{c'}] - [\{*\} \times (-S^{1})]
\end{eqnarray*}
and hence
$$[ \gamma_{c}] + [ \gamma_{c'}] = [K] + [\{*\} \times (-S^{1})] = [K] - [\{*\} \times S^{1}].$$
\begin{figure}[h!]
\begin{center}
 \includegraphics[width = 7cm]{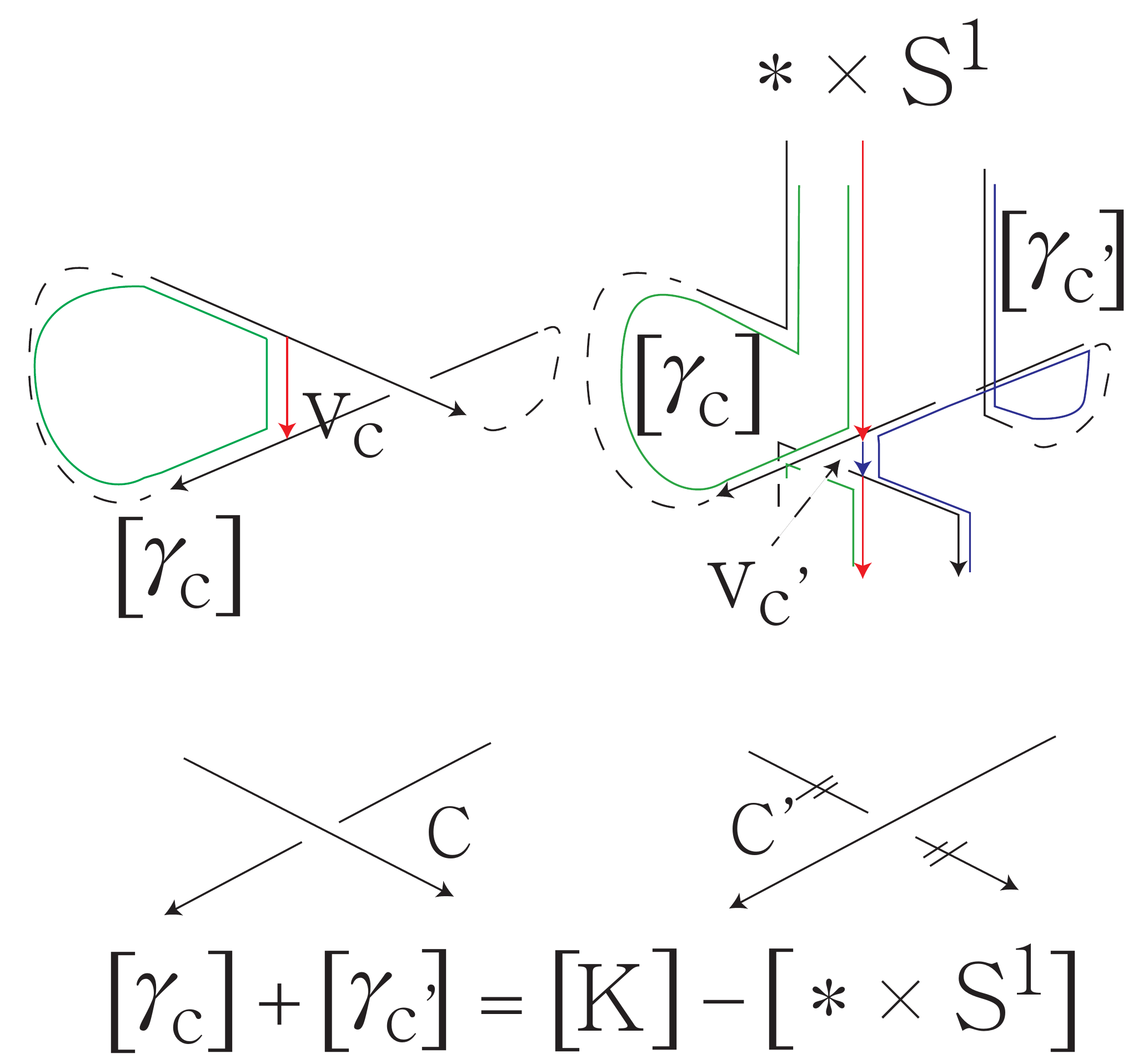}
\end{center}
\caption{Move 4}\label{fig:proof-rotation-move4}
\end{figure}

\end{proof}

\begin{rem}
From the proof of the previous theorem one can see that the fixed element $a$, which is appeared in the condition $i+j =a$ for move 4', is determined by $S^{1}$ which the arc follows along when we do crossing change with two additional double lines. In other words, if we fixed another $a \in H_{1}(S_{g}\times S^{1})/[K]$, then it present that when we do crossing change with two additional double lines, an arc turns around along a closed curve $C$ in $S_{g}\times S^{1}$ such that $[C] = a$ in $H_{1}(S_{g}\times S^{1})/[K]$.
\end{rem}
By the K\"{u}nneth theorem, $H_{1}(S_{g}\times S^{1}) \cong H_{1}(S_{g}) \oplus H_{1}(S^{1})$. Let $p_{1}: H_{1}(S_{g}\times S^{1}) \rightarrow H_{1}(S_{g})$ and $p_{2}: H_{1}(S_{g}\times S^{1}) \rightarrow H_{1}(S^{1})$ be projections. Then we obtain that 
\begin{eqnarray*}
H_{1}(S_{g}\times S^{1})/[K] &\rightarrow& H_{1}(S_{g}\times S^{1})/([p_{1}(K)] \oplus [p_{2}(K)]) \\&\cong& H_{1}(S_{g})/[p_{1}(K)] \oplus H_{1}(S^{1})/[p_{2}(K)]
\end{eqnarray*}
From the previous calculation, one can see the following corollaries.
\begin{cor}
$\{p_{2}([\gamma_{c}])\}$ is the label described in Section 2.
\end{cor}

\begin{cor}
$\{sgn(c)p_{1}([\gamma_{c}])\}$ is an oriented parity.
\end{cor}

\subsection{Universal winding parity}
\begin{dfn}
A parity $p^{w}_{u}$ with coefficients in $A^{w}_{u}$ is called {\em a universal winding parity} if for any winding parity $p^{w}$ with coefficients in $A$ there exists a unique homomorphism of group $\rho:A^{w}_{u} \rightarrow A$ such that $p^{w}_{K} = \rho \circ (p^{w}_{u})_{K}$ for any diagram $K$. 
\end{dfn}

\begin{figure}[h!]
\begin{center}
 \includegraphics[width = 4cm]{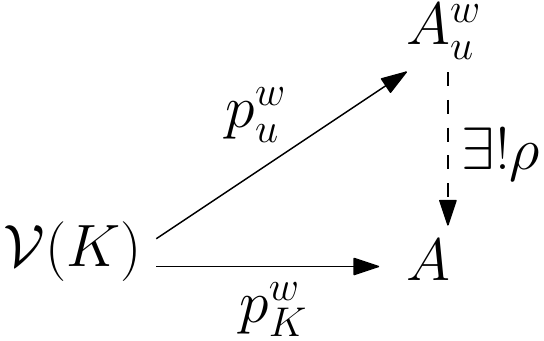}
\end{center}
\end{figure}

Let $K$ be a knot diagram with double lines. Denote by $1_{K,v}$ and  $1_{K,a}$ the generators of the directed summand in the group $\bigoplus_{K}(\bigoplus_{v\in \mathcal{V}(K)} \langle 1_{K,v} 
\rangle   \oplus \langle 1 \rangle )$ corresponding to the vertex $v$.
Let $A^{w}_{u}$ be the group 
$$A^{w}_{u} =  \bigoplus_{K}(\bigoplus_{v\in \mathcal{V}(K)} \langle 1_{K,v} 
\rangle  \oplus \langle 1 \rangle ) /\mathcal{R},$$
where $\mathcal{R}$ is the set of relations of five types.
\begin{enumerate}
\item $1_{K',f_{*}(v)} = 1_{K,v}$ if $v\in \mathcal{V}(K)$ and there exists $f_{*}(v) \in \mathcal{V}(K')$;
\item $1_{K,f_{*}(v_{1})}=1_{K,f_{*}(v_{2})}$ if $f$ is a decreasing second Reidemeister move and $v_{1}$ and $v_{2}$ are the disappearing crossings;
\item $1_{K,f_{*}(v_{1})}-1_{K,f_{*}(v_{2})}+1_{K,f_{*}(v_{3})}= 0$ if $f$ is a third Reidemeister move and $v_{1}, v_{2}, v_{3}$ are the crossing participating in this move such that $v_{1}$ and $v_{3}$ have the middle long arc.
\item $1_{K',f_{*}(v)} = -1_{K,v} +1$ if $f$ is the ``crossing change'' at $v$ along $S^{1}$.
\end{enumerate}

If $p^{w}$ is a parity with coefficients in an abelian group $A$ with fixed element $a \in A$, one define the map $\rho : A_{u} \rightarrow A$ in the following way:
$$\rho(\sum_{K,v \in \mathcal{V}(K)} \lambda_{K,v} 1_{K,v} + \lambda 1) = \sum_{K,v \in \mathcal{V}(K)}  \lambda_{K,v}p^{w}_{v} + \lambda a.$$ 

The universal winding parity is defined analogously to the universal parity, which is introduced in \cite{Manturov_vir_book}. Moreover, a parity, defined by using underlying surfaces of virtual knots, is a universal parity. From this observation the author expects that the homological winding parity gives a universal winding parity.\\[2mm]

{\bf Acknowledgement} The author give thanks to V.O. Manturov, I.M. Nikonov, B. Kim, S. Choi and L. Kauffman for plentiful comments. The author is granted from Russian Foundation for Basic Research (RFBR) grants 20-51-53022 and 19-51-51004.

\end{document}